\documentclass[a4paper,12pt]{article}
\usepackage[margin=1in]{geometry}
\usepackage{amsthm,amssymb,mathtools}
\usepackage{hyperref}

\newtheorem{definition}{Definition}

\newtheorem{corollary}{Corollary}
\newtheorem{theorem}{Theorem}
\newtheorem*{theorem*}{Theorem}
\newtheorem{lemma}{Lemma}

\newtheorem{claim}{Claim}
\newtheorem*{claim*}{Claim}

\newtheorem{observation}{Observation}

\DeclareMathOperator{\tr}{tr}

\begin{document}
\title{A bound on the joint spectral radius using the diagonals}
\author{Vuong Bui\footnote{LIRMM, Universit\'e de Montpellier, CNRS, 161 Rue Ada, 34095 Montpellier, France and UET, Vietnam National University, Hanoi, 144 Xuan Thuy Street, Hanoi 100000, Vietnam  (\href{mailto://bui.vuong@yandex.ru}{\texttt{bui.vuong@yandex.ru}})}}
\date{}
%\subjclass[2020]{15A18, 65F15}

\maketitle
\begin{abstract}
    The primary aim of this paper is to establish bounds on the joint spectral radius for a finite set of nonnegative matrices based on their diagonal elements.
 The efficacy of this approach is evaluated in comparison to existing and related
 results in the field. In particular, let $\Sigma$ be any finite set of $D\times D$ nonnegative matrices with the largest value $U$ and the smallest value $V$ over all positive entries. For each $i=1,\dots,D$, let $m_i$ be any number so that there exist $A_1,\dots,A_{m_i}\in\Sigma$ satisfying $(A_1\dots A_{m_i})_{i,i} > 0$, or let $m_i=1$ if there are no such matrices. We prove that the joint spectral radius $\rho(\Sigma)$ is bounded by
	\begin{equation*}
	\begin{multlined}
    		\max_i \sqrt[m_i]{\max_{A_1,\dots,A_{m_i}\in\Sigma} (A_1\dots A_{m_i})_{i,i}} \le \rho(\Sigma) \\
		\le \max_i \sqrt[m_i]{\left(\frac{UD}{V}\right)^{3D^2} \max_{A_1,\dots,A_{m_i}\in\Sigma} (A_1\dots A_{m_i})_{i,i}}.
	\end{multlined}
	\end{equation*}
	 
  %The corollaries of the approach include simple proofs of some known results, e.g. the joint spectral radius theorem for finite sets of nonnegative matrices.
\end{abstract}

\section{Introduction}
The joint spectral radius is an extension of the ordinary spectral radius of a matrix to a set of matrices. It was first introduced in \cite{rota1960note} by Rota and Strang, and has since caught a lot of attention and become a popular topic in literature. Applications can be found in various engineering fields. The readers can check the book \cite{jungers2009joint} for a good reference for this subject.
The main result in this paper is an effective bound on the joint spectral radius of finite sets of nonnegative matrices. Before presenting it, we give some preliminary definitions.

Given a finite set $\Sigma$ of square matrices with the same dimension, we denote
\[
	\|\Sigma^n\|=\max_{A_1,\dots,A_n\in\Sigma} \|A_1\dots A_n\|,
\]
where $\|A\|$ denotes some norm for a matrix $A$. The article deals with finite sets only for the sake of computational aspects (and certain simplicity). It is left open whether similar results hold for infinite sets.

As the spectral radius of a matrix $A$, which is by definition the maximum of the absolute values of its eigenvalues, can be written in the form of Gelfand's formula as
\[
    \rho(A)=\lim_{n\to\infty} \sqrt[n]{\|A^n\|},
\]
the joint spectral radius of $\Sigma$ is then defined similarly to be the limit
\begin{equation} \label{eq:def-spectral-radius}
	\rho(\Sigma)=\lim_{n\to\infty} \sqrt[n]{\|\Sigma^n\|}.
\end{equation}
The joint spectral radius is also independent of the matrix norm of choice like the ordinary spectral radius. In the text, both radii are denoted by $\rho$, for which the radius to take depends on the argument (or one may simply relate $\rho(A)=\rho(\{A\})$).

The usual approach in proving the limit is to show that $\|\Sigma^{m+n}\|\le\|\Sigma^m\|\|\Sigma^n\|$ for a matrix norm. We remind that a matrix norm is a submultiplicative norm, that is $\|AB\|\le\|A\|\|B\|$ for any two matrices $A,B$ (for example, the Frobenius norm $\|A\|_F=\sqrt{\sum_i\sum_j |A_{i,j}|^2}$, the submultiplicativity is due to Cauchy--Schwarz inequality). As the sequence $\|\Sigma^n\|$ is submultiplicative by
\begin{equation*}
	\begin{multlined}
		\|\Sigma^{m+n}\| = \max_{A_1,\dots,A_{m+n}} \|A_1\dots A_{m+n}\| \le \max_{A_1,\dots, A_{m+n}} \|A_1\dots A_m\| \|A_{m+1}\dots A_{m+n}\| \\
		= \max_{A_1,\dots, A_m} \|A_1\dots A_m\| \max_{A_{m+1},\dots, A_{m+n}} \|A_{m+1}\dots A_{m+n}\| = \|\Sigma^m\| \|\Sigma^n\|,
	\end{multlined}
\end{equation*}
the limit of $\rho(\Sigma)$ exists by Fekete's lemma \cite{fekete1923verteilung}. A consequence is that $\rho(\Sigma)=\inf_n \sqrt[n]{\|\Sigma^n\|}$, that is $\|\Sigma^n\|\ge \rho(\Sigma)^n$ for every $n$.

However, the \emph{maximum norm} will be used throughout the text for convenience, that is $\|A\|=\max_{i,j} |A_{i,j}|$, for which the corresponding lower bound would be 
\begin{equation} \label{eq:lower-bound-maximum-norm}
	\|\Sigma^n\|\ge \frac{1}{D} \rho(\Sigma)^n,
\end{equation}
where $D\times D$ is the dimension of the matrices, since $\|A\|_F\le D\|A\|$ for any matrix $A$.

Beside $\|\Sigma^n\|$, we also denote
\[
	\|\Sigma^n\|_{i,j}=\max_{A_1,\dots,A_n\in\Sigma} |(A_1\dots A_n)_{i,j}|.
\]
The results in this paper address the case of \emph{finite sets of nonnegative matrices} only. The nonnegativity makes the norm to be the maximum entry of the matrix actually. 
A frequently used observation employing the nonnegativity is that 
\[
	\|\Sigma^m\|_{i,k}\|\Sigma^n\|_{k,j}\le \|\Sigma^{m+n}\|_{i,j}\le \sum_{k'} \|\Sigma^m\|_{i,k'}\|\Sigma^n\|_{k',j}\le D\|\Sigma^{m+n}\|_{i,j}
\]
for any $m,n,i,j,k$. (The easy verification is left to the readers.)

Throughout the paper, $\Sigma$ denotes a finite set of nonnegative matrices with dimension $D\times D$ and $U,V$ are respectively the largest value and the smallest value over all the positive entries of the matrices in $\Sigma$. We will \emph{not redefine} them (except only in Section \ref{sec:related}, we redefine $\Sigma$ when discussing related results for other kinds of matrix sets).

Regarding $U,V,D$, we would add an observation that will be used in several proofs: For every $\Delta>0$, every positive entry $x$ of a product of $\Delta$ matrices from $\Sigma$ satisfies\footnote{The readers can easily verify by induction.}
\[
	V^\Delta \le x \le \frac{1}{D} (UD)^\Delta.
\]
It follows that if $\rho(\Sigma)>0$ then 
\[
	V\le\rho(\Sigma)\le UD.
\]

We have introduced the necessary definitions to present the following main result of the paper.
\begin{theorem}\label{thm:jsr-bound}
	Let $m_i$ for each $i$ be any number so that $\|\Sigma^{m_i}\|_{i,i} > 0$, or set $m_i=1$ if there is no such $m_i$. We have
	\[
		\max_i \sqrt[m_i]{\|\Sigma^{m_i}\|_{i,i}} \le \rho(\Sigma) 
		\le \max_i \sqrt[m_i]{\left(\frac{UD}{V}\right)^{3D^2} \|\Sigma^{m_i}\|_{i,i}}.
	\]

\end{theorem}
We discuss how effective these bounds are in Section \ref{sec:related}.

Theorem \ref{thm:jsr-bound} relies on the following theorem, which is not exactly new and can be seen as another perspective on some known results, for which more details are given in Section \ref{sec:related}.
\begin{theorem}\label{thm:formula}
	The joint spectral radius can be written as
	\begin{equation}\label{eq:the-formula}
		\rho(\Sigma) = \sup_n\max_i\sqrt[n]{\|\Sigma^n\|_{i,i}}.
	\end{equation}
	Moreover, if $\rho(\Sigma)>0$, then there exist positive $\alpha,\beta,r$ so that for every $n$,
	\begin{equation}\label{eq:weak-bounds}
		\alpha \rho(\Sigma)^n\le \|\Sigma^n\|\le \beta n^r\rho(\Sigma)^n.
	\end{equation}
\end{theorem}

The main technique in the paper is the following definition of dependency graph.
\begin{definition}
	The \emph{dependency graph} of a set of matrices $\Sigma$ is a directed graph where the vertices are labeled as $1,\dots,D$, and there is an edge from $i$ to $j$ if and only if $A_{i,j}\ne 0$ for some matrix $A\in\Sigma$ (loops are allowed). Since the dependency graph is a directed graph, it can be decomposed into strongly connected components, for which we will call \emph{components} for short. If a component contains only one vertex without any loop, we call it a \emph{single} component. Otherwise, we call it a \emph{regular} component.
\end{definition}
As the dependency graph is the main graph in the article, vertices, paths and components may be mentioned without stating explicitly the dependency graph.
In a similar manner to $\|\Sigma^n\|_{i,j}$, we also denote 
\[
	\|\Sigma^n\|_C=\max_{i,j\in C} \|\Sigma^n\|_{i,j} 
\]
for a component $C$.

The next section discusses additional results and related findings, which the readers may skip if they prefer to focus on the proofs presented in subsequent sections.

\section{Related results and some discussions}
\label{sec:related}
In Theorem \ref{thm:jsr-bound}, although $U/V$ can be arbitrarily large, the appearance of $U$ and $V$ is essential. For example, let $\Sigma$ contain only one matrix 
\begin{equation*}
	A = \begin{bmatrix}
		\frac{1}{N} & 1\\
		1 & \frac{1}{N}
	\end{bmatrix}
\end{equation*}
where $N$ is a large number. As its square is
\begin{equation*}
A^2 = \begin{bmatrix}
  1+\frac{1}{N^2} & \frac{2}{N}\\
  \frac{2}{N} & 1+\frac{1}{N^2}
 \end{bmatrix},
\end{equation*}
the joint spectral radius is greater than $1$, since $\rho(\Sigma)\ge\sqrt{(A^2)_{1,1}}>1$. Setting $m_1=m_2=1$, we have $\sqrt[m_i]{\|\Sigma^{m_i}\|_{i,i}}=\sqrt[1]{\|\Sigma^1\|_{i,i}}=\frac{1}{N}$ for $i=1,2$. Therefore, the relation between $U$ and $V$ must present in the formula in some form.

A related bound in \cite{bui2022joint} has a similar form and is asymptotically equivalent (up to a constant) to the bound in Theorem \ref{thm:jsr-bound}. 
\begin{theorem}\label{thm:using-norm}
	If $\rho(\Sigma)>0$, then for every $n$,
	\[
		\sqrt[n]{\left(\frac{V}{UD}\right)^D \max_C \|\Sigma^n\|_C} \le \rho(\Sigma) \le \sqrt[n]{D \max_C \|\Sigma^n\|_C}.
	\]
\end{theorem}

The performance of the bound in Theorem \ref{thm:using-norm} is rather obvious. 
The ratio between the upper bound and the lower bound is $\sqrt[n]{D\left(\frac{UD}{V}\right)^D}$, for which the interval containing $\rho(\Sigma)$ gets smaller as $n$ increases.

The performance of the bound in Theorem \ref{thm:jsr-bound} is less obvious, although being asymptotically equivalent in the sense that the ratio between the bounds is also at most the root of a constant. At first, we can discard those $i$ with no positive $\|\Sigma^{m_i}\|_{i,i}$. We also make sure that the remaining $m_i$ have large enough values, say they are at least some $m$. Let us say $\max_i\sqrt[m_i]{\|\Sigma^{m_i}\|_{i,i}} = \sqrt[m_j]{\|\Sigma^{m_j}\|_{j,j}}$ and $\max_i \sqrt[m_i]{\left(\frac{UD}{V}\right)^{3D^2} \|\Sigma^{m_i}\|_{i,i}} = \sqrt[m_k]{\left(\frac{UD}{V}\right)^{3D^2} \|\Sigma^{m_k}\|_{k,k}}$, we have
\[
	\sqrt[m_k]{\|\Sigma^{m_k}\|_{k,k}}\le \sqrt[m_j]{\|\Sigma^{m_j}\|_{j,j}}\le \rho(\Sigma)\le \sqrt[m_k]{\left(\frac{UD}{V}\right)^{3D^2} \|\Sigma^{m_k}\|_{k,k}}.
\]
It follows that the ratio between the upper bound and the lower bound in Theorem \ref{thm:jsr-bound} is at most $\sqrt[m_k]{\left(\frac{UD}{V}\right)^{3D^2}}\le \sqrt[m]{\left(\frac{UD}{V}\right)^{3D^2}}$, which is the $m$-th root of a constant. The value of $m$ can be taken arbitrarily large as $\|\Sigma^{tm_i}\|>0$ for any $t\ge 1$. (We note that $m_i$ can be chosen to be the length of a cycle from $i$ to $i$ in the dependency graph.)

Note that $\|\Sigma^n\|_{i,i}$ is computed based on the computation of all $|\Sigma|^n$ combinations. However, it is still reasonable since the problem of approximating the joint spectral radius is $NP$-hard \cite{tsitsiklis1997lyapunov}. The theorem applies however very well if the set contains only one matrix, i.e. the case of the ordinary spectral radius.

In total, both approaches are asymptotically as effective as each other, though one can say that the bound in Theorem \ref{thm:using-norm} is still better by the $n$-th root of a constant.
Furthermore, comparing the bounds in Theorem \ref{thm:jsr-bound} and Theorem \ref{thm:using-norm} would be tricky as the former uses the diagonals while the latter uses the norms. The lower bound in the former is straightforward while the upper bound in the latter is straightforward. Moreover, the estimation of the constants for the latter one is easier, as in \cite{bui2022joint}. However, the former gives some neat formulae with some interesting corollaries (e.g. the proof of the joint spectral radius theorem in this article is simpler than that of \cite{bui2022joint}). 

In Section \ref{sec:duality}, we relate Theorems \ref{thm:jsr-bound} and \ref{thm:using-norm} by showing a way to deduce one theorem from the other using Lemma \ref{lem:not-too-different} (a key lemma in proving Theorem \ref{thm:jsr-bound} itself).

We review some results related to Theorem \ref{thm:formula}.
The first result is for the spectral radius of a (not necessarily nonnegative) matrix by Wimmer.
\begin{theorem*}[Wimmer \cite{wimmer1974spectral}]
	For any complex matrix $A$, we have
	\[
		\rho(A)=\limsup_{n\to\infty} \sqrt[n]{|\tr(A^n)|},
	\]
	where $\tr(A^n)$ denotes the trace of $A^n$.
\end{theorem*}

There is a similar form for the joint spectral radius.
\begin{theorem*}[Chen and Zhou \cite{chen2000characterization}]
	For any finite set $\Sigma$ of complex matrices, we have
	\[
		\rho(\Sigma)=\limsup_{n\to\infty}\max_{A_1,\dots,A_n\in\Sigma} \sqrt[n]{|\tr(A_1\dots A_n)|}.
	\]
\end{theorem*}

We note here that the formula \eqref{eq:the-formula} of $\rho(\Sigma)$ in Theorem \ref{thm:formula} is itself not completely new. Indeed, since for any nonnegative matrices $A_1,\dots,A_n$ we have
\[
    \max_i (A_1\dots A_n)_{i,i}\le \tr(A_1\dots A_n) \le D \max_i (A_1\dots A_n)_{i,i},
\]
it follows from the characterization of Chen and Zhou that
\begin{equation} \label{eq:upper=lower}
    \begin{multlined}
        \limsup_{n\to\infty}\max_{A_1,\dots,A_n\in\Sigma} \sqrt[n]{\max_i (A_1\dots A_n)_{i,i}}\le \rho(\Sigma)\\
        \le \limsup_{n\to\infty}\max_{A_1,\dots,A_n\in\Sigma} \sqrt[n]{D \max_i (A_1\dots A_n)_{i,i}}.
    \end{multlined}
\end{equation}
It means 
\begin{align*}
    \rho(\Sigma)&=\limsup_{n\to\infty}\max_{A_1,\dots,A_n\in\Sigma} \sqrt[n]{\max_i (A_1\dots A_n)_{i,i}}\\
    &=\sup_{n\to\infty}\max_{A_1,\dots,A_n\in\Sigma} \sqrt[n]{\max_i (A_1\dots A_n)_{i,i}}.
\end{align*}
The former equality is due to the lower bound and the upper bound of $\rho(\Sigma)$ in \eqref{eq:upper=lower} being identical since $D$ is a constant and does not affect the asymptotic behavior when $n\to\infty$.
We would explain the latter equality by the following observation, which applies to other places in the text as well: If a nonnegative sequence $x_n$ satisfies $x_{tm}\ge (x_m)^t$ for every $t,m$, then 
\begin{equation} \label{eq:limsup=sup}
    \limsup_{n\to\infty} \sqrt[n]{x_n}=\sup_n \sqrt[n]{x_n}.
\end{equation}
    It suffices to prove $\limsup_{n\to\infty} \sqrt[n]{x_n} \ge \sup_n \sqrt[n]{x_n}$ as the other direction is obvious: 
    For every $m$, we have $\sqrt[m]{x_m}\le \limsup_{n\to\infty} \sqrt[n]{x_n}$ due to the subsequence $\sqrt[tm]{x_{tm}}\ge \sqrt[m]{x_m}$ for $t=1,2,\dots$. It follows that $\sup_m \sqrt[m]{x_m}\le \limsup \sqrt[n]{x_n}$. Equation \eqref{eq:limsup=sup} follows.
    
We also note that the bound \eqref{eq:weak-bounds} of $\|\Sigma^n\|$ in Theorem \ref{thm:formula} is not new either (and not used in the proof of Theorem \ref{thm:jsr-bound}). It is well known in literature, even for bounded sets of complex matrices, e.g. see \cite{bell2005gap}. For finite sets of nonnegative matrices, a stronger bound is already shown in \cite{bui2022joint}: If $\rho(\Sigma)>0$, there exist positive $\alpha,\beta$ and a nonnegative integer $r$ so that for every $n$,
\[
	\alpha n^r \rho(\Sigma)^n\le \|\Sigma^n\|\le \beta n^r\rho(\Sigma)^n.
\]

Although the content of Theorem \ref{thm:formula} is not exactly new as discussed and it can be safely discarded without really affecting Theorem \ref{thm:jsr-bound}, the point of proving it again by combining the formula and the bound\footnote{We note that only the formula \eqref{eq:the-formula} of Theorem \ref{eq:the-formula} is used in the proof of Theorem \ref{thm:jsr-bound}, while the bound \eqref{eq:weak-bounds} of Theorem \ref{thm:jsr-bound} is given as an extra fact.} is that they can be proved together at the same time and in a simple way (about two or three pages without relying on any result). The proof uses a kind of double induction, which may be interesting on its own. 

Moreover, the formula \eqref{eq:the-formula} of Theorem \ref{thm:formula} and the joint spectral radius theorem for finite sets of nonnegative matrices are closely related as follows. Let us first remind the theorem in the original form.
The readers can check \cite[Theorem $2.3$]{jungers2009joint} for an exposition, or \cite{berger1992bounded} for the first time it was proved.

\begin{theorem*}[The joint spectral radius theorem]
	For every bounded set $\Sigma$ of complex matrices, let 
	\[
		\check\rho(\Sigma)=\limsup_{n\to\infty} \sqrt[n]{\check\rho_n(\Sigma)}
	\]
	where $\check\rho_n(\Sigma)=\max_{A_1,\dots,A_n\in\Sigma} \rho(A_1\dots A_n)$, we have 
	\[
		\check\rho(\Sigma)=\rho(\Sigma).
	\]
\end{theorem*}

The quantity $\check\rho(\Sigma)$ is called the \emph{generalized spectral radius}, which may be different from the joint spectral radius in case of an unbounded set of matrices. 

On one hand, suppose we already have the joint spectral radius theorem, in order to prove the formula \eqref{eq:the-formula} of Theorem \ref{thm:formula}, it suffices to prove \eqref{eq:the-formula} for only a nonnegative matrix $A$, that is
\begin{equation} \label{eq:single-matrix}
	\rho(A)=\sup_n\max_i\sqrt[n]{(A^n)_{i,i}},
\end{equation}
and then apply the joint spectral radius theorem to $\Sigma$. Indeed, $\rho(\Sigma)=\check\rho(\Sigma)$ while $\check\rho(\Sigma)$ can be written as

\begin{equation} \label{eq:joint-generalized}
	\begin{split}
		\check\rho(\Sigma)&=\limsup_{n\to\infty} \sqrt[n]{\check\rho_n(\Sigma)}\\
        &=\sup_n \sqrt[n]{\check\rho_n(\Sigma)}\\
        &= \sup_n\max_{A_1,\dots,A_n\in\Sigma} \sqrt[n]{\rho(A_1\dots A_n)} \\
		&= \sup_n\max_{A_1,\dots,A_n\in\Sigma}\sqrt[n]{\sup_t\max_i \sqrt[t]{[(A_1\dots A_n)^t]_{i,i}}} \\
		&= \max_i\sup_t\sup_n \max_{A_1,\dots,A_n\in\Sigma}\sqrt[tn]{[(A_1\dots A_n)^t]_{i,i}} \\
		&=\max_i\sup_n\max_{A_1,\dots,A_n\in\Sigma} \sqrt[n]{(A_1\dots A_n)_{i,i}} \\
        &=\sup_n\max_i\sqrt[n]{\|\Sigma^n\|_{i,i}}.
	\end{split}
\end{equation}
Note that the second equality is due to \eqref{eq:limsup=sup} since we have $\check\rho_{tm}(\Sigma)\ge (\check\rho_{m}(\Sigma))^t$ for any $t,m$, which is verified by
\begin{align*}
(\check\rho_m(\Sigma))^t&=\left(\max_{A_1,\dots,A_m\in\Sigma} \rho(A_1\dots A_m)\right)^t\\
&= \left(\max_{A_1,\dots,A_m\in\Sigma} \lim_{k\to\infty} \sqrt[k]{\|(A_1\dots A_m)^k\|}\right)^t\\
&= \left(\max_{A_1,\dots,A_m\in\Sigma} \lim_{k'\to\infty} \sqrt[k't]{\|(A_1\dots A_m)^{k't}\|}\right)^t\\
&\le \max_{B_1,\dots,B_{tm}\in\Sigma} \lim_{k'\to\infty} \sqrt[k']{\|(B_1\dots B_{tm})^{k'}\|}\\
&=\check\rho_{tm}(\Sigma).
\end{align*}
The forth equality is by \eqref{eq:single-matrix}. We can verify the sixth equality: By considering $t=1$, we obtain
\[
    \max_i\sup_t\sup_n \max_{A_1,\dots,A_n\in\Sigma}\sqrt[tn]{[(A_1\dots A_n)^t]_{i,i}} \ge \max_i\sup_n\max_{A_1,\dots,A_n\in\Sigma} \sqrt[n]{(A_1\dots A_n)_{i,i}},
\]
while the other direction of the above inequality is due to for every $t,n$,
\[
 	\max_{A_1,\dots,A_n\in\Sigma}\sqrt[tn]{[(A_1\dots A_n)^t]_{i,i}} \le \max_{B_1,\dots,B_{tn}\in\Sigma}\sqrt[tn]{(B_1\dots B_{tn})_{i,i}}.
\]

Although the formula \eqref{eq:the-formula} of Theorem \ref{thm:formula} (for a set of matrices) can be reduced to \eqref{eq:single-matrix}, which is \eqref{eq:the-formula} for a single matrix, proving Theorem \ref{thm:formula} for a set of matrices is not harder than proving for a single matrix. Therefore, we provide the full proof without relying on the joint spectral radius theorem.

On the other hand, suppose we have the formula \eqref{eq:the-formula} of Theorem \ref{thm:formula} in the first place. This leads to a simple proof for the joint spectral radius theorem for finite sets of nonnegative matrices, since it follows from \eqref{eq:joint-generalized} and \eqref{eq:the-formula} that
\[
    \check\rho(\Sigma)=\sup_n\max_i\sqrt[n]{\|\Sigma^n\|_{i,i}}=\rho(\Sigma).
\]
Note that \eqref{eq:joint-generalized} depends on \eqref{eq:single-matrix}, which is a special case of \eqref{eq:the-formula}. However, \eqref{eq:joint-generalized} is independent of the joint spectral radius theorem.
\begin{corollary}
	The joint spectral radius theorem holds for finite sets of nonnegative matrices.
\end{corollary}

In comparison to the standard proof of the joint spectral radius theorem by Elsner in \cite{elsner1995generalized}, which makes use of analytic geometric tools, our combinatorial approach seems to be more elementary and asks for less background on the subject. However, we should emphasize here that our approach works for only finite sets of nonnegative matrices.

Following the development of Wimmer's and Chen and Zhou's results, Xu has attempted to turn the limit superior in the theorem of Chen and Zhou to a limit for nonnegative matrices with the condition that one of them is primitive. Note that a matrix $A$ is primitive if $A^n>0$ for some $n\ge 1$.
\begin{theorem*}[Xu \cite{xu2010trace}]
	For any finite set of nonnegative matrices $\Sigma$ with at least one primitive matrix, we have
	\[
		\rho(\Sigma)=\lim_{n\to\infty}\max_{A_1,\dots,A_n\in\Sigma} \sqrt[n]{\tr(A_1\dots A_n)}.
	\]
\end{theorem*}

Using the formula in Theorem \ref{thm:formula}, we can extend Xu's result to the following theorem, by showing that the conclusion still holds with a more relaxed condition than the primitivity of at least one matrix.

\begin{theorem} \label{thm:extension-xu}
	Given a finite set of nonnegative matrices $\Sigma$, for each $i$, denote
	\[
		\delta_i=\gcd\{n: \|\Sigma^n\|_{i,i}>0\},
	\]
	or we set $\delta_i=1$ in case the set is empty.
	Let $\Delta$ be a multiple of all $\delta_i$. We have
	\[
		\rho(\Sigma)=\lim_{m\to\infty}\max_i \sqrt[m\Delta]{\|\Sigma^{m\Delta}\|_{i,i}} = \lim_{m\to\infty} \max_{A_1,\dots,A_{m\Delta}} \sqrt[m\Delta]{\tr(A_1\dots A_{m\Delta})}.
	\]
\end{theorem}
Note that a matrix $A$ is primitive if and only if it is irreducible, i.e. all $\delta_i$ corresponding to $\Sigma=\{A\}$ are equal to $1$. If such a matrix $A$ is an element of $\Sigma$, we still have the same values $\delta_i=1$ for all $i$, for which we have the conclusion for $\Delta=1$, i.e. Xu's result.

In fact, the proof of Theorem \ref{thm:extension-xu} in Section \ref{sec:extension-xu} can be seen as a corollary of Theorem \ref{thm:formula}, with the support of a variant of Fekete's lemma for nonnegative sequences, which might be interesting on its own. 

\section{Proof of Theorem \ref{thm:jsr-bound}}
\label{sec:jsr-bound}
To prove Theorem \ref{thm:jsr-bound}, we need the following key lemma, which is fairly technical, and will be proved in Section \ref{sec:key-lemma}. The lemma is itself used to relate Theorems \ref{thm:jsr-bound} and \ref{thm:using-norm} in Section \ref{sec:duality}.

\begin{lemma}
\label{lem:not-too-different}
For any index $i$, if $m,n$ are two positive integers whose difference is bounded, then either $\|\Sigma^n\|_{i,i}=0$ or the ratio $\|\Sigma^m\|_{i,i}/\|\Sigma^n\|_{i,i}$ is bounded. In particular, if $m\ge n$, then the bound can be set to
\[
	(UD)^{m-n} \left(\frac{UD}{V}\right)^{3D^2-2D+1}.
\] 
\end{lemma}
We are interested in an explicit constant only for the case $m\ge n$ since it will be applied in the proof of Theorem \ref{thm:jsr-bound} as follows. When $m<n$, we do not need the boundedness of the ratio for other results, but still prove it as an interesting fact.

\begin{proof}[Proof of Theorem \ref{thm:jsr-bound}]
	The lower bound is obvious, we prove the upper bound.

	When we consider $\|\Sigma^n\|_C$ for a component $C$ instead of considering $\|\Sigma^n\|$, we are actually considering the problem reduced to $C$ in the sense that we remove all the dimensions not in $C$. This is still a problem that satisfies all the results of the original problem. Therefore, the limit 
	\[
		\rho_C(\Sigma) = \lim_{n\to\infty} \sqrt[n]{\|\Sigma^n\|_C}
	\]
	can be written by Theorem \ref{thm:formula} as
	\[
		\rho_C(\Sigma) = \sup_n \max_{i\in C} \sqrt[n]{\|\Sigma^n\|_{i,i}},
	\]
	which means
	\[
		\rho(\Sigma)=\max_C \rho_C(\Sigma).
	\]

	Consider a regular component $C$ and any $i\in C$, by \eqref{eq:lower-bound-maximum-norm} (in the introduction) we have
	\[
		\rho_C(\Sigma)^{m_i}\le D \|\Sigma^{m_i}\|_C.
	\]
	Suppose $\|\Sigma^{m_i}\|_C=\|\Sigma^{m_i}\|_{j,k}$. Let $\delta_1$ be the distance from $i$ to $j$ and $\delta_2$ be the distance from $k$ to $i$, that is $\|\Sigma^{\delta_1}\|_{i,j}$ and $\|\Sigma^{\delta_2}\|_{k,i}$ are both nonzero (when $i=j$, we have $\delta_1=0$, we then assume that $\|\Sigma^0\|_{i,j}=1$, and similarly for $j=k$ with $\delta_2=0$ and $\|\Sigma^0\|_{k,i}=1$). We have
	\begin{equation*}
    	\begin{multlined}
    		\|\Sigma^{m_i}\|_C = \frac{1}{\|\Sigma^{\delta_1}\|_{i,j} \|\Sigma^{\delta_2}\|_{k,i}} \|\Sigma^{\delta_1}\|_{i,j}\|\Sigma^{m_i}\|_{j,k} \|\Sigma^{\delta_2}\|_{k,i}\\
		\le \frac{1}{V^{\delta_1+\delta_2}} \|\Sigma^{m_i+\delta_1+\delta_2}\|_{i,i} \le \frac{1}{V^{\delta_1+\delta_2}} (UD)^{\delta_1+\delta_2} \left(\frac{UD}{V}\right)^{3D^2-2D+1} \|\Sigma^{m_i}\|_{i,i}\\
		= \left(\frac{UD}{V}\right)^{3D^2-2D+1+\delta_1+\delta_2} \|\Sigma^{m_i}\|_{i,i} \le \left(\frac{UD}{V}\right)^{3D^2-1} \|\Sigma^{m_i}\|_{i,i},
    	\end{multlined}
	\end{equation*}
	where Lemma \ref{lem:not-too-different} is used in the step bounding $\|\Sigma^{m_i+\delta_1+\delta_2}\|_{i,i}$ by a constant times $\|\Sigma^{m_i}\|_{i,i}$. (Note that $3D^2-2D+1+\delta_1+\delta_2\le 3D^2-1$ is due to $\delta_1\le D-1, \delta_2\le D-1$.)

	In total, 
	\[
		\rho_C(\Sigma)\le\sqrt[m_i]{D\left(\frac{UD}{V}\right)^{3D^2-1} \|\Sigma^{m_i}\|_{i,i}}\le \sqrt[m_i]{\left(\frac{UD}{V}\right)^{3D^2} \|\Sigma^{m_i}\|_{i,i}}.
	\]

	When $C$ is single, $C$ contains a single vertex without any loop, hence $\rho_C(\Sigma)=0$, and the above inequality trivially holds with $m_i=1$. (In fact, $m_i$ can be set to any value, we set $m_i=1$ for simplicity.)

	As $\rho(\Sigma)=\max_C \rho_C(\Sigma)$, we obtain the conclusion
	\[
		\rho(\Sigma)\le\max_i \sqrt[m_i]{\left(\frac{UD}{V}\right)^{3D^2}\|\Sigma^{m_i}\|_{i,i}}.\qedhere
	\]
\end{proof}

\section{Proof of Lemma \ref{lem:not-too-different}}
\label{sec:key-lemma}
It suffices to consider only $i$ for which the set $\{n: \|\Sigma^n\|_{i,i}>0\}$ is nonempty.
To prove Lemma \ref{lem:not-too-different}, we need the following lemma.
\begin{lemma}
\label{lem:avail-for-large}
Let $d=\gcd\{n: (\Sigma^n)_{i,i}>0\}$. There exists $N$ so that $\|\Sigma^n\|_{i,i}>0$ for every $n\ge N$ with $d|n$. 
In particular, one can set $N=(D-1)(2D-1)$.
\end{lemma}

Note that Lemma \ref{lem:avail-for-large} is asymptotically optimal in the worst case. For example, if the dependency graph is composed of only two disjoint cycles around $i$ of lengths $\ell_1,\ell_2$ so that $\ell_1,\ell_2$ are not too much distant with $\gcd(\ell_1,\ell_2)=1$, then the smallest number in the place of $N$ would be $(\ell_1-1)(\ell_2-1)$. (We remind that the Frobenius number\footnote{The Frobenius number of positive integers $p_1,\dots,p_k$ with $\gcd(p_1,\dots,p_k)=1$ is the largest integer that cannot be expressed as a linear combination of $p_1,\dots,p_k$ with nonnegative coefficients.} of $x,y$ is $(x-1)(y-1)-1$.) An example of $\ell_1,\ell_2$ when $D=2k$ is $\ell_1=k+1$ and $\ell_2=k$ (note that $\ell_1+\ell_2=D+1$). When $D$ is odd, we leave a vertex isolated and proceed with the even number of remaining vertices.

To prove Lemma \ref{lem:avail-for-large}, we need some preliminary results.

Two subwalks in a walk are said to be \emph{disjoint} if the only vertex that they possibly share is a common endpoint.
\begin{observation}
\label{obs:less-than-2D}
If a walk does not contain two disjoint circuits, then its length is less than $2D$.
\end{observation}
\begin{proof}
    Let the walk be $v_0,\dots,v_k$. If the vertices are all distinct, then $k<D$ and we are done. Otherwise, let $v_j$ be the first repeated vertex, that is, $j$ is the smallest number so that there exists $i<j$ with $v_i=v_j$. It follows that $v_0,\dots,v_{j-1}$ are distinct. Suppose the walk does not contain two disjoint circuits. This mean there is no circuit in the walk $v_j,\dots,v_k$, that is, these vertices are distinct. It follows that $k < 2D$.
\end{proof}

We also need Schur's lemma\footnote{The lemma is due to Schur in 1935 but was not published until 1942 by Brauer in \cite{brauer1942problem}.} that gives a bound on the Frobenius number.
\begin{lemma}[Schur 1935 \cite{brauer1942problem}]
Let $2\le p_1<p_2<\dots<p_k$ be $k$ integers such that $\gcd(p_1,\dots,p_k)=1$, then every integer $n\ge(p_1-1)(p_k-1)$ can be expressed as a linear combination of $p_1,\dots,p_k$ with nonnegative coefficients.
\end{lemma}

We can now prove Lemma \ref{lem:avail-for-large}.
\begin{proof}[Proof of Lemma \ref{lem:avail-for-large}]
Denote $S=\{n: \|\Sigma^n\|_{i,i}>0\}$ and $d=\gcd S$. Let $m$ be the smallest element of $S$ such that the set $S^*=\{n\in S: n\le m\}$ satisfies $\gcd\,S^*=d$. 

We prove that $m<2D$. Indeed, suppose $m\ge 2D$.
Due to the minimality of $m$, we have $d^*=\gcd\,(S^*\setminus\{m\}) > d$ and $d^*$ does not divide $m$.
As $\|\Sigma^m\|_{i,i}>0$, there is a circuit from $i$ to $i$ of length $m$. This circuit contains $2$ disjoint subcircuits, by Observation \ref{obs:less-than-2D}. 
Let $a$ and $b$ be the lengths of the two subcircuits. Removing any of these subcircuits or both results in a circuit of length less than $m$, which is divisible by $d^*$. In other words, $d^* \mid m-a$, $d^* \mid m-b$ and $d^*\mid m-a-b$. It implies that $d^*\mid (m-a) + (m-b) - (m-a-b) = m$, contradiction.

Let $T=\{n/d: n\in S^*\}$, we have $\gcd\,T=1$.
As $\min S^*\le D$ (a minimal circuit) and $\max S^*\le m<2D$, we have $\min T\le D/d$ and $\max T\le 2D/d$.
	It follows that $\|\Sigma^n\|_{i,i} > 0$ for every $n\ge d(\frac{D}{d}-1)(\frac{2D}{d}-1)$ and $n$ divisible by $d$, by Schur's lemma. The conclusion follows as $d(\frac{D}{d}-1)(\frac{2D}{d}-1)\le (D-1)(2D-1)$.
\end{proof}

Finally comes the proof of Lemma \ref{lem:not-too-different}.
\subsection*{Proof of Lemma \ref{lem:not-too-different}}
Denote $d=\gcd\{t: \|\Sigma^t\|_{i,i}>0\}$.

Denote $N=(D-1)(2D-1)$. By Lemma \ref{lem:avail-for-large}, for every $t\ge N$, if $d|t$ then $\|\Sigma^t\|_{i,i}>0$. 

Let $S_t=\{j: \|\Sigma^t\|_{j,i}>0\}$ for each $t\ge 1$, that is the set of vertices from which we can reach $i$ by a walk of length $t$.
Let $S=\bigcup_{t: d|t} S_t$, that is the set of vertices from which we can reach $i$ by a walk of some length divisible by $d$.

Let $M= N+D^2$. We have the following observation.
\begin{claim*}
    $S_t=S$ for every $t\ge M$ with $d|t$.
\end{claim*}
\begin{proof}
    For each $j\in S$, let $ud$ be the least multiple of $d$ such that $\|\Sigma^{ud}\|_{j,i}>0$. We have $u\le D$. Indeed, suppose $u>D$. Consider the path $v_0,\dots,v_{ud}$ from $v_0=j$ to $v_{ud}=i$. The vertices $v_{kd}$ for $k=0,\dots,u$ are not all distinct since $u>D$, say $v_{k'd}=v_{k''d}$. Contracting the subpath $v_{k'd},\dots,v_{k''d}$ from the path $v_0,\dots,v_{ud}$ gives a path from $j$ to $i$ whose length is divisible by $d$ but less than $ud$, contradiction.
    
	Consider any $t=vd\ge M$. It follows from $M=N+D^2\ge N+ud$ that $t-ud=vd-ud\ge N$, which implies $\|\Sigma^{vd-ud}\|_{i,i} > 0$ by Lemma \ref{lem:avail-for-large}. That is $\|\Sigma^t\|_{j,i} \ge \|\Sigma^{ud}\|_{j,i} \|\Sigma^{vd-ud}\|_{i,i} > 0$, i.e. the vertex $j$ is also in $S_t$.
\end{proof}

	Note that $M=N+D^2=(D-1)(2D-1)+D^2=3D^2-3D+1$, for which one can observe
	\[
		M+d=3D^2-3D+1+d\le 3D^2-2D+1.
	\]
	We consider the case $\|\Sigma^m\|_{i,i}>0$ and $\|\Sigma^n\|_{i,i}>0$ only, as the conclusion is trivial otherwise. It follows that $d$ divides both $m,n$. Cases regarding the magnitude of $m,n$ are analyzed as follows:
	\begin{itemize}
		\item Suppose $n\ge M+d$ and $m> d$. It follows that there exists a greatest positive integer $t<\min\{m,n\}$ with $n-t\ge M$ and $d|t$. Such a number $t$ exists because $t=d$ is a satisfying number. Since $d$ divides $m-t$, we have
	\begin{align*}
		\|\Sigma^m\|_{i,i} &\le \sum_{j\in S} \|\Sigma^t\|_{i,j} \|\Sigma^{m-t}\|_{j,i} \\
		&\le \sum_{j\in S} \|\Sigma^t\|_{i,j} \|\Sigma^{n-t}\|_{j,i} \frac{\max_{j\in S} \|\Sigma^{m-t}\|_{j,i}}{\min_{j\in S} \|\Sigma^{n-t}\|_{j,i}} \\
		&\le  \frac{\max_{j\in S} \|\Sigma^{m-t}\|_{j,i}}{\min_{j\in S} \|\Sigma^{n-t}\|_{j,i}} D \|\Sigma^n\|_{i,i}.
	\end{align*}
Note that denominator is positive as $n-t\ge M$ and $d\mid n-t$.

			Since $m-t$ and $n-t$ are bounded (as $t$ is chosen to be the greatest satisfying number), the ratio $\max_{j\in S} \|\Sigma^{m-t}\|_{j,i}/\min_{j\in S} \|\Sigma^{n-t}\|_{j,i}$ is bounded, and so is the ratio $\|\Sigma^m\|_{i,i}/\|\Sigma^n\|_{i,i}$.

			Suppose $n-m\le M$. Now comes the explicit bound (the situation $m\ge n$ is included in this case). We have $n-t\le M+d$. Indeed, if $n-t > M+d$, then $t+d < n-M \le m$, $t+d<n-M<n$ and $n-(t+d) > M$, a contradiction to the maximality of $t$ (since $t+d$ is a larger satisfying number than $t$). It follows that $n-t\le M+d\le 3D^2 - 2D + 1$.
Therefore, 
\[
	\begin{multlined}
	\frac{\|\Sigma^m\|_{i,i}}{\|\Sigma^n\|_{i,i}} \le D\frac{\max_{j\in S} \|\Sigma^{m-t}\|_{j,i}}{\min_{j\in S} \|\Sigma^{n-t}\|_{j,i}} \le D \frac{\frac{1}{D}(UD)^{m-t}}{V^{n-t}} \\
		= \left(\frac{UD}{V}\right)^{n-t} (UD)^{m-n}\le \left(\frac{UD}{V}\right)^{3D^2-2D+1} (UD)^{m-n}.
	\end{multlined}
\]

\item
It remains to consider the case 
   we do not have both $m>d$ and $n\ge M+d$. 
			If $m=d$ then $\|\Sigma^m\|_{i,i}/\|\Sigma^n\|_{i,i}$ is bounded as the range of $n$ is bounded when $m-n$ is bounded (we do not need an explicit bound here as $m<n$). If $n< M+d\le 3D^2-2D+1$, then regardless of $m$ we still have
\[
	\frac{\|\Sigma^m\|_{i,i}}{\|\Sigma^n\|_{i,i}} \le \frac{\frac{1}{D} (UD)^m}{V^n} \le \left(\frac{UD}{V}\right)^n (UD)^{m-n} \le \left(\frac{UD}{V}\right)^{3D^2-2D+1} (UD)^{m-n}.
\]
	\end{itemize}
The conclusion follows from the verification in both cases.

\section{Proof of Theorem \ref{thm:formula}}
\label{sec:formula}
If $\rho(\Sigma)=0$, then there is no cycle in the dependency graph, hence the theorem trivially holds. We assume $\rho(\Sigma)>0$.

Denote
\[
    \lambda=\sup_n\max_i\sqrt[n]{\|\Sigma^n\|_{i,i}}.
\]
We have
\[
    \rho(\Sigma)=\lim_{n\to\infty} \sqrt[n]{\|\Sigma^n\|}\ge\lambda,
\]
since $\|\Sigma^{tn}\|\ge \|\Sigma^{tn}\|_{i,i}\ge (\|\Sigma^n\|_{i,i})^t$ for any $t,n$. (The latter inequality is obtained by induction with $\|\Sigma^{tn}\|_{i,i}\ge \|\Sigma^{(t-1)n}\|_{i,i} \|\Sigma^{n}\|_{i,i}$.)

To finish the proof, it suffices to prove the existence of $\beta,r$ so that for every $n$,
\begin{equation} \label{eq:polynomial-upper-bound}
	\|\Sigma^n\|\le\beta n^r\lambda^n.
\end{equation}
Indeed, suppose we have \eqref{eq:polynomial-upper-bound}. We then have
\[
    \lambda\le\rho(\Sigma)=\lim_{n\to\infty} \sqrt[n]{\|\Sigma^n\|}\le \lim_{n\to\infty} \sqrt[n]{\beta n^r\lambda^n}=\lambda.
\]
We obtain the equality $\rho(\Sigma)=\lambda$, which is the formula \eqref{eq:the-formula} of Theorem \ref{thm:formula}. The bounds \eqref{eq:weak-bounds} of Theorem \ref{thm:formula} also follow, since
\[
    \|\Sigma^n\|\le\beta n^r\lambda^n=\beta n^r\rho(\Sigma)^n,
\]
while the lower bound $\|\Sigma^n\|\ge \alpha \rho(\Sigma)^n$ for $\alpha=\frac{1}{D}$ is due to by \eqref{eq:lower-bound-maximum-norm} (in the introduction).

The reduction to \eqref{eq:polynomial-upper-bound} is clarified. We now prove \eqref{eq:polynomial-upper-bound}.

We begin with an observation: There exists some $K_0$ so that $\|\Sigma^n\|_{i,j}\le K_0 \lambda^n$ for every $n$ and every $i,j$ in a regular component. Indeed, let $\delta$ be the distance from $j$ to $i$ in the dependency graph, we have
\begin{equation}\label{eq:same-component}
	\|\Sigma^n\|_{i,j}=\frac{1}{\|\Sigma^\delta\|_{j,i}} \|\Sigma^n\|_{i,j} \|\Sigma^\delta\|_{j,i} \le \frac{1}{\|\Sigma^\delta\|_{j,i}} \|\Sigma^{n+\delta}\|_{i,i}\le \frac{1}{\|\Sigma^\delta\|_{j,i}} \lambda^{n+\delta}\le K_0\lambda^n,
\end{equation}
where
\[
	K_0=\max_{j',i'} \frac{\lambda^{\delta'}}{\|\Sigma^{\delta'}\|_{j',i'}} 
\]
over all pairs $j',i'$ so that there is path from $j'$ to $i'$, and the distance from $j'$ to $i'$ is $\delta'$.

Note that if the component is single (containing only one vertex $i$ without any loop), then the inequality $\|\Sigma^n\|_{i,i}\le K_0 \lambda^n$ in \eqref{eq:same-component} still trivially holds with $\|\Sigma^n\|_{i,i}=0$ for every $n$.

Now we may wonder what would be the inequality when $i,j$ are not in the same component. 

The condensation\footnote{The condensation of a directed graph $G$ is the directed graph whose vertices are the strongly connected components of $G$ and there is an edge from $U$ to $V$ if there is an edge from $u$ to $v$ in $G$ with $u\in U$ and $v\in V$.} of the dependency graph gives us the relation between the components. We denote by $\Delta(i,j)$ the distance from the component of $i$ to the component of $j$. (For $i,j$ in the same component, we let $\Delta(i,j)=0$, and we do not consider $i,j$ with no path from $i$ to $j$.)
For any $\delta$, we also denote
\[
	\|\Sigma^n\|_\delta=\max_{i,j:\ \Delta(i,j) \le \delta} \|\Sigma^n\|_{i,j}.
\]

What we have shown in \eqref{eq:same-component} is actually $\|\Sigma^n\|_0\le K_0 \lambda^n$. It is the base case of the following claim.

\begin{claim}
    For every $\delta$, there exist a positive constant $K_\delta$ and a nonnegative number $r_\delta$ so that for every $n$,
    \[
    	\|\Sigma^n\|_\delta\le K_\delta n^{r_\delta}\lambda^n.
    \]
\end{claim}

\begin{proof}
    By the induction method, since we have established the claim for $\delta=0$ with $K_0$ in \eqref{eq:same-component} and $r_0=0$, it remains to show the claim for any $\delta$, given that it holds for $\delta'=\delta-1$ (with the corresponding numbers $K_{\delta'}, r_{\delta'}$). 
    
    Let $\gamma=\max\{\frac{\|\Sigma^1\|}{\lambda}, D(K_{\delta'})^2\}$ and $H=\max\{1,K_0D, 2^{2r_{\delta'}}\}$. It suffices to show that for every $n$,
    \begin{equation} \label{eq:explicit-form}
        \|\Sigma^n\|_{\delta} \le \gamma H^{\lceil \log n\rceil} \lambda^n,
    \end{equation}
    since $H^{\lceil \log n\rceil} \le H^{1+\log n}= HH^{\log n} = Hn^{\log H}$ (note that $H\ge 1$),    
    which implies
    \[
        \|\Sigma^n\|_{\delta} \le \gamma Hn^{\log H} \lambda^n.
    \]
   (In other words, this is the claim for $\delta$ with $K_\delta=\gamma H$ and $r_\delta=\log H$.)
   
    In order to prove \eqref{eq:explicit-form}, we again use another induction on $\lceil \log n\rceil$ as follow (i.e. double induction, first on $\delta$, then on $\lceil \log n\rceil$). (Notation $\log n$ here denotes the logarithm of base $2$.) 
    At first, the base case trivially holds for those $n$ with $\lceil \log n\rceil = 0$, i.e. $n=1$. That is because $\|\Sigma^1\|_\delta\le \|\Sigma^1\|=\frac{\|\Sigma^1\|}{\lambda} \lambda \le \gamma\lambda$ while $H^{\lceil \log n\rceil}=1$.
    
    We assume \eqref{eq:explicit-form} holds for any number $n'$ so that $\lceil \log n'\rceil < \lceil \log n\rceil$ and proves it also holds for $n$. Let $n=\ell+m$ where $\ell=\lfloor n/2\rfloor$ and $m=\lceil n/2\rceil$. Suppose $\|\Sigma^n\|_\delta=\|\Sigma^n\|_{i,j}$ for some $i,j$. We have
    \[
	    \|\Sigma^n\|_{i,j}\le\sum_{k'} \|\Sigma^\ell\|_{i,k'} \|\Sigma^m\|_{k',j}\le D \|\Sigma^\ell\|_{i,k} \|\Sigma^m\|_{k,j}
    \]
    for some $k$ that maximizes $\|\Sigma^\ell\|_{i,k} \|\Sigma^m\|_{k,j}$.
   
	We consider three cases regarding $k$:
	\begin{itemize}
		\item If $i,k$ are in the same component then $\|\Sigma^\ell\|_{i,k}\le K_0 \lambda^\ell$ by \eqref{eq:same-component}, and $\|\Sigma^m\|_{k,j}\le \|\Sigma^m\|_\delta \le\gamma H^{\lceil \log m\rceil} \lambda^m$ by the induction hypothesis on $\lceil \log n\rceil$, since $\lceil \log m\rceil =\lceil \log \lceil n/2\rceil \rceil = \lceil \log n\rceil - 1$. It follows that
    \[
    	\|\Sigma^n\|_{i,j}\le D K_0 \lambda^\ell \gamma H^{\lceil \log m\rceil} \lambda^m \le H\gamma H^{\lceil \log m\rceil} \lambda^n = \gamma H^{\lceil \log n\rceil} \lambda^n,
    \]
    where the latter inequality is due to $K_0 D\le H$.
    
    \item If $k,j$ are in the same component then likewise we have
    \[
    	\|\Sigma^n\|_{i,j}\le D \gamma H^{\lceil \log \ell\rceil} \lambda^\ell  K_0 \lambda^m \le \gamma H^{\lceil \log n\rceil} \lambda^n.
    \]
    Note that $H H^{\lceil \log \ell\rceil}\le H H^{\lceil \log m\rceil} = H^{\lceil \log n\rceil}$ since $H\ge 1$ (the inequality may be strict, say when $\lceil \log \ell\rceil < \lceil \log m\rceil$, e.g. $n=2^t+1$, for which we need $H\ge 1$).
    
    \item If $k$ is not in the same component with either $i$ or $j$, then both $\Delta(i,k)$ and $\Delta(k,j)$ are at most $\delta'=\delta-1$. It follows that $\|\Sigma^\ell\|_{i,k}\le\|\Sigma^\ell\|_{\delta'}\le K_{\delta'} \ell^{r_{\delta'}}\lambda^\ell$ and $\|\Sigma^m\|_{k,j}\le\|\Sigma^m\|_{\delta'}\le K_{\delta'} m^{r_{\delta'}}\lambda^m$ by the induction hypothesis on $\delta$. We have
	    \[
            \|\Sigma^n\|_{i,j}\le D K_{\delta'} \ell^{r_{\delta'}}\lambda^\ell K_{\delta'} m^{r_{\delta'}}\lambda^m\le D (K_{\delta'})^2 n^{2r_{\delta'}} \lambda^n \le \gamma H^{\lceil \log n\rceil} \lambda^n,
    \]
    where the last inequality is due to the condition $\gamma\ge D(K_{\delta'})^2$ and $H^{\lceil\log n\rceil}\ge H^{\log n} = n^{\log H} \ge n^{2r_{\delta'}}$ (note that we have the conditions $H\ge 1$ and $H\ge 2^{2r_{\delta'}}$).
	\end{itemize}

    We have verified \eqref{eq:explicit-form} by induction on $\lceil \log n\rceil$. The claim follows, by induction on $\delta$.
\end{proof}

Since every $\Delta(i,j)$ is less than $D$, it follows that $\|\Sigma^n\|=\|\Sigma^n\|_D$. Therefore, the claim implies \eqref{eq:polynomial-upper-bound}, which in turn finishes the proof of Theorem \ref{thm:formula}.

\section{Equivalence of Theorem \ref{thm:jsr-bound} and Theorem \ref{thm:using-norm} up to a constant using Lemma \ref{lem:not-too-different}}
\label{sec:duality}
We present a way to deduce each of Theorems \ref{thm:jsr-bound} and \ref{thm:using-norm} from the other, using Lemma \ref{lem:not-too-different}. However, the version of Theorem \ref{thm:using-norm} that is deduced from Theorem \ref{thm:jsr-bound} is obtained with the weaker constant $\left(\frac{V}{UD}\right)^{3D^2}$ instead of $\left(\frac{V}{UD}\right)^D$.

Given some $i$ in some component $C$, suppose $\|\Sigma^n\|_C=\|\Sigma^n\|_{j,k}$, and let $\ell_1,\ell_2$ be the distance from $i$ to $j$ and from $k$ to $i$, respectively, we have
\begin{equation}\label{eq:norm-to-diagonal}
	\|\Sigma^n\|_C =\|\Sigma^n\|_{j,k}= \frac{1}{\|\Sigma^{\ell_1}\|_{i,j}\|\Sigma^{\ell_2}\|_{k,i}} \|\Sigma^{\ell_1}\|_{i,j}\|\Sigma^n\|_{j,k}\|\Sigma^{\ell_2}\|_{k,i} \le \frac{1}{V^{\ell_1+\ell_2}}\|\Sigma^{n+\ell_1+\ell_2}\|_{i,i}.
\end{equation}
Note that if $i=j$ (resp. $k=i$), then $\ell_1=0$ (resp. $\ell_2=0$) and we assume $\|\Sigma^0\|_{i,j}=1$ (resp. $\|\Sigma^0\|_{k,i}=1$).

The following proof shares some parts with the proof in Section \ref{sec:jsr-bound}.
\begin{proof}[Deduction of Theorem \ref{thm:jsr-bound} from Theorem \ref{thm:using-norm}]
	As the lower bound of Theorem \ref{thm:jsr-bound} is trivial, we prove the upper bound.

	Fix a regular component $C$ and choose any $i\in C$. Since $\|\Sigma^{m_i}\|_{i,i}>0$, it follows from \eqref{eq:norm-to-diagonal} with $n=m_i$ and Lemma \ref{lem:not-too-different} that
	\begin{align*}
		\|\Sigma^{m_i}\|_C&\le \frac{1}{V^{\ell_1+\ell_2}} (UD)^{\ell_1+\ell_2} \left(\frac{UD}{V}\right)^{3D^2-2D+1} \|\Sigma^{m_i}\|_{i,i} \\
	&\le \left(\frac{UD}{V}\right)^{3D^2-2D+1+\ell_1+\ell_2} \|\Sigma^{m_i}\|_{i,i} \\
	&\le \left(\frac{UD}{V}\right)^{3D^2-1} \|\Sigma^{m_i}\|_{i,i},
\end{align*}
since $\ell_1\le D-1, \ell_2\le D-1$.

	Multiplying by $D$ and taking the root, we obtain
	\[
		\sqrt[m_i]{D\|\Sigma^{m_i}\|_C} \le \sqrt[m_i]{D\left(\frac{UD}{V}\right)^{3D^2-1}\|\Sigma^{m_i}\|_{i,i}}\le \sqrt[m_i]{\left(\frac{UD}{V}\right)^{3D^2}\|\Sigma^{m_i}\|_{i,i}}.
	\]

    If $C$ is a single component, that is $C$ contains a single vertex $i$ with no loop, then the above inequality trivially holds with all sides being zeros.

	Since $\rho(\Sigma) \le \sqrt[n]{D \max_C \|\Sigma^n\|_C}$ by Theorem \ref{thm:using-norm}, it follows that
	\[
		\rho(\Sigma) \le \max_i \sqrt[m_i]{\left(\frac{UD}{V}\right)^{3D^2}\|\Sigma^{m_i}\|_{i,i}}.\qedhere
	\]
\end{proof}

\begin{proof}[Deduction of Theorem \ref{thm:using-norm} (with a weaker constant) from Theorem \ref{thm:jsr-bound}]
	As the upper bound of Theorem \ref{thm:using-norm} is trivial, we prove the lower bound. We start with \eqref{eq:norm-to-diagonal}:
	\[
		\|\Sigma^n\|_C\le\frac{1}{V^{\ell_1+\ell_2}}\|\Sigma^{n+\ell_1+\ell_2}\|_{i,i}.
	\]

	Suppose $C$ is regular, which means both sides of the above inequality are positive. Let $\delta\le D$ be the length of the shortest cycle from $i$ to $i$. Since $\|\Sigma^{t\delta}\|_{i,i}>0$ for any $t\ge 1$, two consecutive elements in the set $\{\ell:\|\Sigma^\ell\|_{i,i}>0\}$ have the distance at most $\delta\le D$. Therefore, if $n+\ell_1+\ell_2 > D$, there exists some positive integer $m<n+\ell_1+\ell_2$ with $n+\ell_1+\ell_2-m\le D$ so that $\|\Sigma^m\|_{i,i}>0$. We first consider the case $n+\ell_1+\ell_2 > D$, it follows from Lemma \ref{lem:not-too-different} that
	\begin{align*}
		\|\Sigma^n\|_C&\le\frac{1}{V^{\ell_1+\ell_2}} (UD)^{n+\ell_1+\ell_2-m} \left(\frac{UD}{V}\right)^{3D^2-2D+1} \|\Sigma^m\|_{i,i}\\
		&\le (UD)^{n-m} \left(\frac{UD}{V}\right)^{3D^2-2D+1+\ell_1+\ell_2} \|\Sigma^m\|_{i,i}\\
		&\le (UD)^{n-m} \left(\frac{UD}{V}\right)^{3D^2-2D+1+\ell_1+\ell_2} \rho(\Sigma)^m\\
		&\le (UD)^{n-m} \left(\frac{UD}{V}\right)^{3D^2-2D+1+\ell_1+\ell_2} \rho(\Sigma)^n \rho(\Sigma)^{m-n}\\
		&\le \left(\frac{UD}{\rho(\Sigma)}\right)^{n-m} \left(\frac{UD}{V}\right)^{3D^2-2D+1+\ell_1+\ell_2} \rho(\Sigma)^n.
	\end{align*}

	If $n\ge m$, it follows from $\rho(\Sigma)\ge V$ that
	\begin{align*}
		\|\Sigma^n\|_C&\le \left(\frac{UD}{V}\right)^{n-m} \left(\frac{UD}{V}\right)^{3D^2-2D+1+\ell_1+\ell_2} \rho(\Sigma)^n\\
		&= \left(\frac{UD}{V}\right)^{3D^2-2D+1+\ell_1+\ell_2+n-m} \rho(\Sigma)^n\\
		&\le \left(\frac{UD}{V}\right)^{3D^2-D+1} \rho(\Sigma)^n.\\
	\end{align*}

	When $n<m$, it follows from $\rho(\Sigma)\le UD$ that
	\begin{align*}
		\|\Sigma^n\|_C&\le \left(\frac{UD}{V}\right)^{3D^2-2D+1+\ell_1+\ell_2} \rho(\Sigma)^n\\
		&\le \left(\frac{UD}{V}\right)^{3D^2-1} \rho(\Sigma)^n,
	\end{align*}
	since $\ell_1\le D-1, \ell_2\le D-1$.

	Combining the bounds on $\|\Sigma^n\|_C$ for $n\ge m$ and $n<m$ gives
	\[
		\|\Sigma^n\|_C\le \left(\frac{UD}{V}\right)^{3D^2} \rho(\Sigma)^n.
	\]

	In the remaining case that $n+\ell_1+\ell_2\le D$, we also have
	\[
		\|\Sigma^n\|_C\le\|\Sigma^n\|\le \frac{1}{D} (UD)^n \le \left(\frac{UD}{V}\right)^n V^n \le \left(\frac{UD}{V}\right)^{3D^2} \rho(\Sigma)^n,
	\]
	since $n\le n+\ell_1+\ell_2\le D\le 3D^2$.

	When $C$ is single, the above inequality trivially holds with $\|\Sigma^n\|_C=0$.

	Taking the maximum over all the components, we obtain the conclusion
	\[
		\rho(\Sigma)\ge \sqrt[n]{\left(\frac{V}{UD}\right)^{3D^2}\max_C \|\Sigma^n\|_C}.\qedhere
	\]
\end{proof}

\section{Proof of Theorem \ref{thm:extension-xu}}
\label{sec:extension-xu}
Before proving Theorem \ref{thm:extension-xu}, we give the following lemma, which is an extension of Fekete's lemma for nonnegative sequences.
\begin{lemma}
\label{lem:nonnegative-fekete}
	Given a nonnegative sequence $a_n$ for $n=1,2,\dots$, if the sequence is supermultiplicative, that is, $a_{i+j}\ge a_ia_j$ for every $i,j\ge 1$, then the subsequence of all positive $\sqrt[n]{a_n}$ is either empty or converges to $\sup_n \sqrt[n]{a_n}$.
\end{lemma}
\begin{proof}
    Suppose the subsequence of all positive $\sqrt[n]{a_n}$ is not empty. Let $m$ be so that $a_m>0$, we have the subsequence of all positive $\sqrt[n]{a_n}$ is infinite ($a_{mt}>0$ for every $t\ge 1$).
    
    It is obvious by definition that $\limsup\{\sqrt[n]{a_n}: a_n>0\} \le \sup_n \sqrt[n]{a_n}$. To finish the proof, it remains to show $\liminf\{\sqrt[n]{a_n}: a_n>0\} \ge \sup_n \sqrt[n]{a_n}$.
    
    Consider any positive integer $q$.
    Let $R$ be the set of integers $r$ ($0\le r< q$) such that there exists some $m_r$ with $m_r\equiv r \pmod{q}$ and $a_{m_r} > 0$. For each $r\in R$, we denote by $m_r$ the smallest such number.
    
    For every $n$ such that $a_n>0$, if $n\equiv r \pmod{q}$, then $r\in R$. Consider the representation $n=pq+m_r$, we obtain 
    \[
        \sqrt[n]{a_n} \ge \sqrt[n]{(a_q)^p a_{m_r}}.
    \]
    The lower bound converges to $\sqrt[q]{a_q}$ when $n\to\infty$ since $m_r$ is bounded.
    
    As the lower bound holds for every $q$, we have shown the lower bound $\sup_n \sqrt[n]{a_n}$ for the limit inferior and finished the proof.
\end{proof}

We can now prove Theorem \ref{thm:extension-xu}
\begin{proof}[Proof of Theorem \ref{thm:extension-xu}]
	It suffices to prove the first equality since the second equality is quite obvious by the fact that $\frac{1}{D} \tr(A)\le \max_i A_{i,i} \le \tr(A)$ for any $D\times D$ nonnegative matrix $A$.

	At first, since the sequence $\{\|\Sigma^n\|_{i,i}\}_n$ for each $i$ is supermultiplicative, it follows from the extension of Fekete's lemma in Lemma \ref{lem:nonnegative-fekete} that the subsequence of all the positive elements is either empty or follows the growth rate 
	\[
		\rho_i=\lim\left\{\sqrt[n]{\|\Sigma^n\|_{i,i}}: \|\Sigma^n\|_{i,i}>0\right\}=\sup_n \sqrt[n]{\|\Sigma^n\|_{i,i}}.
	\]
	In case of emptiness, we still have $\rho_i=0$ as the value of the supremum.

	By Theorem \ref{thm:formula}, we have
	\[
		\rho(\Sigma)=\max_i\rho_i. 
	\]

	Further, by the definition of $\delta_i$, if the set $\{n:\|\Sigma^n\|_{i,i}>0\}$ is nonempty, then for every large enough multiple $n$ of $\delta_i$, we have $\|\Sigma^n\|_{i,i}>0$ (Lemma \ref{lem:avail-for-large} even gives an example of how large $n$ is), that is
	\[
		\rho_i=\lim_{m\to\infty} \sqrt[m\delta_i]{\|\Sigma^{m\delta_i}\|_{i,i}}.
	\]
	When the set is empty, it trivially holds with $\delta_i=1$.

	Since $\Delta$ is a multiple of all $\delta_i$, it follows that 
	\[
		\rho(\Sigma)=\max_i \rho_i=\max_i \lim_{m\to\infty} \sqrt[m\Delta]{\|\Sigma^{m\Delta}\|_{i,i}}=\lim_{m\to\infty} \max_i \sqrt[m\Delta]{\|\Sigma^{m\Delta}\|_{i,i}}.\qedhere
	\]
\end{proof}

\section*{Acknowledgements}
The author would like to thank G\"unter Rote for his various useful remarks. It should be also noted that the anonymous reviewer has given very helpful suggestions, particularly on the presentation of the work.

\bibliographystyle{plain}

\bibliography{diajsr}

\begin{thebibliography}{10}

\bibitem{bell2005gap}
Jason~P Bell.
\newblock A gap result for the norms of semigroups of matrices.
\newblock {\em Linear Algebra and its Applications}, 402:101--110, 2005.

\bibitem{berger1992bounded}
Marc~A Berger and Yang Wang.
\newblock Bounded semigroups of matrices.
\newblock {\em Linear Algebra and its Applications}, 166:21--27, 1992.

\bibitem{brauer1942problem}
Alfred Brauer.
\newblock On a problem of partitions.
\newblock {\em American Journal of Mathematics}, 64(1):299--312, 1942.

\bibitem{bui2022joint}
Vuong Bui.
\newblock On the joint spectral radius of nonnegative matrices.
\newblock {\em Linear Algebra and its Applications}, 654:89--101, 2022.

\bibitem{chen2000characterization}
Quande Chen and Xinlong Zhou.
\newblock Characterization of joint spectral radius via trace.
\newblock {\em Linear Algebra and its Applications}, 315(1-3):175--188, 2000.

\bibitem{elsner1995generalized}
Ludwig Elsner.
\newblock The generalized spectral-radius theorem: an analytic-geometric proof.
\newblock {\em Linear Algebra and its Applications}, 220:151--159, 1995.

\bibitem{fekete1923verteilung}
Michael Fekete.
\newblock {\"U}ber die {V}erteilung der {W}urzeln bei gewissen algebraischen
  {G}leichungen mit ganzzahligen {K}oeffizienten.
\newblock {\em Mathematische Zeitschrift}, 17(1):228--249, 1923.

\bibitem{jungers2009joint}
Rapha{\"e}l Jungers.
\newblock {\em The joint spectral radius: theory and applications}, volume 385.
\newblock Springer Science \& Business Media, 2009.

\bibitem{rota1960note}
Gian–Carlo Rota and W.~{Gilbert Strang}.
\newblock A note on the joint spectral radius.
\newblock {\em Indagationes Mathematicae (Proceedings)}, 63:379--381, 1960.

\bibitem{tsitsiklis1997lyapunov}
John~N Tsitsiklis and Vincent~D Blondel.
\newblock {The Lyapunov exponent and joint spectral radius of pairs of matrices
  are hard—when not impossible—to compute and to approximate}.
\newblock {\em Mathematics of Control, Signals and Systems}, 10(1):31--40,
  1997.

\bibitem{wimmer1974spectral}
H.~K. Wimmer.
\newblock Spectral radius and radius of convergence.
\newblock {\em The American Mathematical Monthly}, 81(6):625--627, 1974.

\bibitem{xu2010trace}
Jianhong Xu.
\newblock On the trace characterization of the joint spectral radius.
\newblock {\em The Electronic Journal of Linear Algebra}, 20:367--375, 2010.

\end{thebibliography}
\end{document}